\documentclass[12pt]{amsart}
\usepackage{amsmath, epsfig, amssymb, qtree}
\DeclareMathAlphabet{\curly}{U}{rsfs}{m}{n}

\newtheorem{thm}{Theorem}
\newtheorem{lem}{Lemma}[section]
\theoremstyle{definition}

\newtheorem{conj}{Conjecture}

\theoremstyle{remark}
\theoremstyle{plain}

\renewcommand{\b}{\ensuremath{\beta}}
\newcommand{\del}{\ensuremath{\delta}}
\newcommand{\lam}{\ensuremath{\lambda}}
\newcommand{\eps}{\ensuremath{\varepsilon}}

\newcommand{\pfrac}[2]{{\left(\frac{#1}{#2}\right)}}

\def\({\left(}
\def\){\right)}

\newcommand{\be}{\begin{equation}}
\newcommand{\ee}{\end{equation}}
\newcommand{\benn}{\begin{equation*}}   % no label equation displays
\newcommand{\eenn}{\end{equation*}}
\numberwithin{equation}{section}

\newif\ifdraft
\begin{document}

\title
{The number of solutions of $\lambda(x)=n$}

\author{Kevin Ford and Florian Luca}
\address{KF: Department of Mathematics, 1409 West Green Street, University
of Illinois at Urbana-Champaign, Urbana, IL 61801, USA}
\email{ford@math.uiuc.edu}

\address{FL : Instituto de Matem{\'a}ticas,
 Universidad Nacional Autonoma de M{\'e}xico,
C.P. 58089, Morelia, Michoac{\'a}n, M{\'e}xico}
\email{fluca@matmor.unam.mx }

\date{\today}
%\thanks{2000 Mathematics Subject Classification: Primary 11N25;
%  Secondary 62G30}
\thanks{Research of the first author supported by National Science 
Foundation grants DMS-0555367 and DMS-0901339. Research of the second author was supported in part by Grants SEP-CONACyT 79685 and PAPIIT
100508.}

\begin{abstract} We study the question of whether for each $n$ there is an
$m\ne n$ with $\lambda(m)=\lambda(n)$, where $\lambda$ is 
Carmichael's function. We give a ``near" proof  of the fact that this is the case unconditionally, and 
a complete conditional proof under the Extended Riemann Hypothesis.
\end{abstract}

\maketitle

\centerline{\it To Professor Carl Pomerance on his $65$th birthday}

%
%%%%%%%%%%%%%%%%%%%%%%%%%%%%%%%%%%%%%%%%%%%%%
%
\section{Introduction}\label{sec:intro}
%
%%%%%%%%%%%%%%%%%%%%%%%%%%%%%%%%%%%%%%%%%%%%%

Let $\lambda(n)$ be the Carmichael function, that is, $\lambda(n)$ is
the largest order of any number modulo $n$.
Recently, Banks et al \cite{BFLPS} made the following conjecture:

\begin{conj}\label{lambda}
For every positive integer $n$, there is an integer $m\ne n$ with
$\lambda(m)=\lambda(n)$. 
\end{conj}

The analogous question for the Euler function
$\phi(n)$ is known as Carmichael's conjecture and remains unsolved.
If there are counterexamples to Conjecture \ref{lambda}, 
the authors of \cite{BFLPS}
proved that all such counterexamples $n$ are multiples of the smallest
counterexample $n_0$.  Further, they showed that if $n_0$ exists, then $n_0$
is divisible by every prime less than $30000$.
In this note, we prove that Conjecture \ref{lambda} follows from the Extended 
Riemann Hypothesis (ERH) for Dirichlet $L$-functions, and also we come very
close to proving the conjecture unconditionally.

If $n$ has prime factorization $n=p_1^{e_1} \cdots p_k^{e_k}$, then
$\lambda(n) = [\lambda(p_1^{e_1}),\ldots,\lambda(p_k^{e_k})]$, where
$[a_1,\ldots,a_k]$ denotes the least common multiple of $a_1,\ldots,a_k$,
$\lambda(p^e)=p^{e-1}(p-1)$ when $p$ is odd or $e\le 2$, and
$\lambda(2^e)=2^{e-2}$ when $e\ge 3$.
The following is proved in \S 7 of \cite{BFLPS}.

\begin{lem}\label{lem1}
Suppose $n_0$ exists, that is, Conjecture \ref{lambda} is false.  
Then (i) $2^4|n_0$ and
(ii) if $(p-1)|\lambda(n_0)$ for a prime $p$, then $p^2|n_0$.
\end{lem}

\begin{proof}
Since $\lambda(1)=\lambda(2)$ and $\lambda(4)=\lambda(8)$,
part (i) follows.  If $(p-1)|\lambda(n_0)$ and $p\nmid n_0$, then
$\lambda(n_0)=\lambda(pn_0)$, which proves that $p|n_0$.  
Assume that $p^2\nmid n_0$.
By the minimality of $n_0$, $\lambda(n_0/p)=\lambda(m)$ for some $m\ne n_0/p$.
We have $p\nmid m$, else $(p-1)|\lambda(n_0/p)$ and
$\lambda(n_0)=\lambda(n_0/p)$.  Thus,
$$
\lambda(n_0)=[p-1,\lambda(n_0/p)] = [p-1,\lambda(m)] = \lambda(pm),
$$
a contradiction.  Therefore, $p^2 | n_0$, proving (ii).
\end{proof}

With Lemma \ref{lem1}, it is easy to show that many primes must divide
$n_0$.  For example, by (i) and (ii) with $p=3$ and $p=5$, 
we immediately
obtain $3^2|n_0$ and $5^2|n_0$.  Thus, $2^2 \cdot 3\cdot 5 |\lambda(n_0)$,
and by (ii) again, $n_0$ is divisible by $7^2$, $11^2$, $13^2$, 
$31^2$ and $61^2$.   
Subject to certain hypotheses, we may continue this
process and deduce that every prime must divide $n_0$, which would prove
Conjecture \ref{lambda}.

\bigskip

{\bf Notation.}  Throughout, the letters $p,q,r,s$, with or without
subscripts, will always denote primes.  By \emph{prime power} we mean
a number of the form $p^a$ where $p$ is prime and $a\ge 1$, and a
\emph{proper prime power} is a prime power with $a\ge 2$.

\bigskip

For a prime $q$, we construct
a tree $T(q)$ with $q$ as the root node as follows.  Below $q$ form links
 to each prime power $p^e$ with $p^e \| (q-1)$.  Now continue the process,
linking each $p^e$ to the prime powers $r^b$ with $r^b \| (p-1)$, etc.
The end result will be a tree with leaves
which are powers of 2.  For example, here is the tree corresponding to
$q=149$.

\medskip

\Tree [.149 [.$2^2$ ] [.37 [.$2^2$ ] [.$3^2$ [.2 ] ] ] ]

\medskip

Let $f(q)$ denote the
largest proper prime power occurring in the tree.  Set $f(q)=1$ if
there are no proper prime powers in the tree; this only happens when
$q\in \{2,3,7,43\}$ 
(If $q$ is the smallest prime $>43$ with $f(q)=1$, then $q-1$ is squarefree and
 $q > 2\cdot 3 \cdot 7\cdot 43 + 1$ by explicit calculation,
so $q-1$ has a prime divisor $r$ other than
$2,3,7,43$.  By the minimality of $q$, $f(r)>1$ and therefore
$f(q)>1$, a contradiction).  Alternatively,
we may define $f(q)$ inductively by the formulas $f(2)=1$ and if
$q\ge 3$ and $q-1=p_1^{e_1} \cdots p_k^{e_k}$ with
$e_1=\cdots=e_h=1<e_{h+1}\le e_{h+2} \le \cdots \le e_k$, then
$$
f(q) = \max ( f(p_1), \ldots, f(p_h), p_{h+1}^{e_{h+1}}, \ldots ,p_k^{e_k} ).
$$
For example, $f(149)=9$.  The tree $T(q)$ is similar to the tree
constructed for the Pratt primality certificate \cite{Pratt}.

\begin{conj}\label{conjB}
For every prime power $p^a$, there is a prime $q$ with $p^a |(q-1)$ and
$f(q)<p^{a+1}$.
\end{conj}

Note that we must have $p^a \| (q-1)$.

\begin{thm}\label{thm2}
Conjecture \ref{conjB} implies Conjecture \ref{lambda}.
\end{thm}

\begin{proof}
Suppose Conjecture \ref{conjB} is true and Conjecture \ref{lambda} is false.
Let $p^{a+1}$ be the smallest prime power not dividing
$\lambda(n_0)$ (here $a\ge 0$).  Each prime power divisor of $p-1$ is 
$<p^{a+1}$ and hence $(p-1)|\lambda(n_0)$.
Lemma \ref{lem1} implies that $p^2|n_0$, thus $p|\lambda(n_0)$ and $a\ge 1$.
Let $b=a+1$ if $p>2$ and $b=a+2$ if $p=2$, so that
$\lambda(p^b)=p^a(p-1)$.  We have $p^b\| n_0$, since $p^{b+1}|n_0$ implies
$p^{a+1}|\lambda(n_0)$ and $p^b \nmid n_0$ implies $\lambda(n_0)=
\lambda(pn_0)$.  We next claim that
every prime $r$ with $f(r)<p^{a+1}$ satisfies $r^2|n_0$.
Proceed by induction on $r$, noting that the case $r=2$ is taken care of by
Lemma \ref{lem1} (i).  Suppose $s>2$, $f(s)<p^{a+1}$
and every prime $r<s$ with $f(r)<p^{a+1}$ 
satisfies $r^2|n_0$.  If $r \| (s-1)$,
then   $f(r)<p^{a+1}$ and hence $r|\lambda(n_0)$, and if $r^c \| (s-1)$
with $c\ge 2$ then $r^c < p^{a+1}$ and hence $r^c|\lambda(n_0)$.
Consequently, $(s-1)|\lambda(n_0)$, and applying Lemma \ref{lem1} once again
we see that $s^2|n_0$.
By hypothesis, there is a prime $q$ with $p^a | (q-1)$ and $f(q)<p^{a+1}$.
In particular, $q^2|n_0$ and $q|\lambda(n_0)$.
This means $p^a | \lambda(n_0/p^b)$ and
$$
\lam(n_0)=[\lam(p^b), \lam(n_0/p^b)] = [\lam(p^{b-1}), \lam(n_0/p^b)]
= \lam(n_0/p),
$$
a contradiction.
\end{proof}

\bigskip

We pose the following questions. (1) 
For each proper prime power $p^a$, is there a prime $q$ with $f(q)=p^a$ ?  
(2) Is there a prime power $p^a$ so that there are infinitely many
primes $q$ with $f(q)=p^a$ ?
(3) Does $f(q)\to \infty$ as $q\to \infty$?  Computations suggest that
there are infinitely many primes $q$ with $f(q)=4$, but this will be very
difficult to prove.

It is clear that $f(q)$ is at most the largest prime power dividing 
$q-1$, thus 
\be\label{fqp}
p^a \| (q-1) \; \text{ and } \; q < p^{2a+1} \implies f(q)<p^{a+1}.
\ee
Hence, it is almost
sufficient to find a prime $q\equiv 1\pmod{p^a}$ with $q < (p^a)^{2+1/a}$.
Let $P(b,m)$ denote the least prime which is $\equiv b\pmod{m}$.
Linnik proved that there is a constant $L$ such that $P(b,m)\ll m^L$
for all coprime $b,m$.  The best constant known today is $L=5.5$
and due to Heath-Brown.   However, the Extended Riemann Hypothesis (ERH)
for Dirichlet $L$-functions implies that
\be\label{ERH}
\left| \pi(x,m,b) - \frac{\text{li}(x)}{\phi(m)}\right| \le x^{1/2} \log (xm^2)
\ee
uniformly in $x,m,b$ \cite{Oes}, where $\pi(x,m,b)$ is the number
of primes $r\le x$ with $r\equiv b\pmod{m}$ and
$\text{li}(x)=\int_2^x \frac{dt}{\log t} \sim
\frac{x}{\log x}$.   Consequently, we may take $L=2+\eps$ for any fixed 
$\eps$.  Using \eqref{ERH} and a finer analysis of $f(q)$,
we prove the following.

\begin{thm}\label{thm:ERH}
ERH implies Conjecture \ref{conjB}.
\end{thm}

The main result of this paper is the following ``near'' proof of
Conjecture \ref{conjB}.  

\begin{thm}\label{thm3}
For an effective constant $K$, if $p^a>K$ then there is 
a prime $q$ with $p^a |(q-1)$ and $f(q)<p^{a+1}$.
\end{thm}

Theorem \ref{thm3} is proved in the next section.  Next, the proof of 
Theorem \ref{thm:ERH} will be given in Section \ref{sec:ERH}.

%%%%%%%%%%%%%%%%%%%%%%%%%%%%%%%%%%%%%%%%%%%%%%%%%%%%%%
%
%
\section{Proof of Theorem \ref{thm3}}\label{sec:thm3}
%
%
%%%%%%%%%%%%%%%%%%%%%%%%%%%%%%%%%%%%%%%%%%%%%%%%%%%%%%

We need first an effective lower bound for the number of primes in 
an arithmetic progression with prime power modulus.

\begin{lem}\label{pix}
There are positive, effective constants 
$K_1, K_2, K_3$ so that if $p^a \ge K_1$
and $x\ge p^{aK_2}$, then 
$$
\pi(x;p^a,1) - \pi(x;p^{a+1},1) \ge K_3 \frac{x/\log x}{p^{a+1/2}\log p}.
$$
\end{lem}

\begin{proof}
This basically follows from an effective version of Linnik's Theorem.
For a modulus $q \ge 3$, let $\beta=\beta(q)$ the largest real zero of an
$L$-function (primitive or not) of a real character of modulus $q$.  
If no such zero exists, set
$\beta=\frac12$.   By Prop. 18.5 of \cite{IK}, there are effective constants 
$c_1,c_2,c_3$ so that if $x \ge q^{c_1}$ then
\be\label{Psixq1}
\Psi(x;q,1) = \frac{x}{\phi(q)} \left[ 1 - \frac{x^{\beta-1}}{\beta} +
\theta\( x^{-\eta}+ \frac{\log q}{q} \) \right],
\ee
where $|\theta| \le c_2$ and
$$
\eta = \eta(q) = \frac{c_3 \log (2 + \frac{2}{(1-\beta)\log q})}{\log q}.
$$
If $p>2$, then the real character modulo $p^a$ has conductor
$p$, hence $\beta(p^a)=\beta(p)$.
If $p=2$ then any real character modulo $p^a$ has conductor
4 or 8 and $\beta(2^a)=\frac12$.
By a classical theorem \cite[\S 14 (12)]{Da}, there is an effective constant $c>0$
so that we have
$$
\beta(p^a) \le 1 - \frac{c}{p^{1/2}\log^2 p}.
$$
Fix a prime power $p^a \ge 8$ and let $\b=\b(p)$, $\eta=\eta(p^a)$.
By \eqref{Psixq1} with $q=p^a$ and with $q=p^{a+1}$, we have
\be\label{psi2}
\Psi(x;p^a,1)-\Psi(x;p^{a+1},1) = \frac{x}{p^a} \left[ 1 -
  \frac{x^{\beta(p)-1}}{\beta(p)} + \theta' \( 
x^{-\eta}+ \frac{\log p^a}{p^a} \) \right],
\ee
where $|\theta'| \le c_2 \frac{p+1}{p-1} \le 3 c_2$.
If $\beta \le 1 - 1/\log p^a$, then the left side of \eqref{psi2} is
$\ge x/(2p^a)$ if $p^a$ and $K_2$ are sufficiently large.
If $\b > 1 - 1/\log p^a$, let $\del=1-\beta$, so that
\begin{align*}
1-\frac{x^{\b-1}}{\b} \ge \b - x^{-\del} &\ge 1-\del - e^{-\del K_2 \log p^a}
\\
&\ge -\del + \frac{\del K_2 \log p^a}{1+\del K_2 \log p^a}
\ge \del \( -1 + \frac{K_2\log p^a}{1+K_2} \) \\
&\ge \frac{K_2}{2+2K_2} (\del \log p^a)
\end{align*}
and
$$
x^{-\eta}\le \pfrac{\del \log p^a}{2}^{c_3 K_2} \le 2^{-K_2 c_3} (\del \log
p^a). 
$$
Hence,
$$
\Psi(x;p^a,1) -\Psi(x;p^{a+1},1) \gg 
\frac{x}{p^a} (\del \log p^a) \gg \frac{x}{p^{a+1/2} \log  p}.
$$
Finally,
$$
\pi(x;p^a,1) -\pi(x;p^{a+1},1) \ge
\frac{\Psi(x;p^a,1)-\Psi(x;p^{a+1},1) -O(\sqrt{x})}{\log x}
$$
and the proof is complete.
\end{proof}

Our next tool is an upper bound for the number of \emph{prime chains}
of a certain type.  A \emph{prime chain} is a sequence $p_1,\ldots,p_k$
of primes such that $p_i|(p_{i+1}-1)$ for $1\le i\le k-1$.
The following is Theorem 2 in \cite{FKL}.
%%%%%%%%%%%%%%%%%%%%%%%%%%%%%%%%%%%
%
% Chains lemma
%
%%%%%%%%%%%%%%%%%%%%%%%%%%%%%%%%%%%

\begin{lem}\label{chains}
For every $\eps>0$ there is an effective constant $C(\eps)$ so that
for any prime $p$, the number of prime chains with $p_1=p$ and
$p_k\le x$ (varying $k$) is $\le C(\eps) (x/p)^{1+\eps}$. 
\end{lem}

{\bf Remark.}  At the moment, the method of \cite{FKL} gives
$$
C(\eps) = \exp \exp \( (1+o(1)) \frac{1}{\eps} \log \frac{1}{\eps} \)
$$
as $\eps\to 0^+$.  We need a numerical value of $C(\eps)$ in one case.
By the argument in
\S 3 of \cite{FKL}, if $y<p$, $w$ is the product of the primes $\le y$,
and $s>1$, then the number of primes in question is at most the largest column
sum of
\[
x^s \sum_{0\le k\le \frac{\log x}{\log 2}} M^{k}, \quad
M = \biggl( \sum_{\substack{m\ge 1 \\ am+1\equiv b(\!\bmod{w})}}
m^{-s} \biggr)_{b,a\in (\mathbb{Z}/w\mathbb{Z})^*}.
\]
If all the eigenvalues of $M$ lie inside the unit circle, then
$\sum_{k=0}^\infty M^k = (I-M)^{-1}$. 
For example, taking $s=\frac54$ and $w=210$, so that $M$ is a $48
\times 48$ matrix, we compute that the largest column sum of
$(I-M)^{-1}$ is $\le 7.37$, so $C(\frac14)=7.37$ is admissible.

\begin{lem}\label{fqlarge}
For $0<\eps\le 1$ and $y\ge 10^{10}$, we have
$$
\# \{ q \le x : f(q) \ge y \} \le \frac{c(\eps) x^{1+\eps}}{
  y^{1/2+\eps}\log y}, 
$$
where
$$
c(\eps) = C(\eps)(2^{-1-\eps}-6^{-1-\eps}) \zeta(1+\eps) \( 0.44 +
\frac{2.43}{1+2\eps} \).
$$
\end{lem}

\begin{proof}
For a prime power $s^b \ge y$ with $b\ge 2$, 
let $q$ be a prime with $f(q)=s^b$.
Then there is a prime $r\equiv 1\pmod{s^b}$
and a prime chain with $p_1=r$ and $p_k=q$. 
Write $r=k s^b + 1$. 
By Lemma \ref{chains}, the number of such $q\le x$ is at most
$$
\sum_{\substack{r\le x \\ r\equiv 1\pmod{s^b}}} C(\eps)
\pfrac{x}{r}^{1+\eps} \le C(\eps) \pfrac{x}{s^b}^{1+\eps}
\sum_{\substack{k\ge 1 \\ ks^b+1 \text{ prime}}} k^{-1-\eps}. 
$$
 If $s>3$, we note that $k$ is even and among any 
three consecutive even values of $k$, $r$ is prime for at most two of
them.  For such $s$, the sum on $k$ is at most
$(2^{-1-\eps}-6^{-1-\eps})\zeta(1+\eps)$.  For $s\in \{2,3\}$, we bound
the sum on $k$ trivially as $\zeta(1+\eps)$.
The number of $q\le x$ is therefore at most
\be\label{ceps}
C(\eps) x^{1+\eps} \zeta(1+\eps) \left[ \sum_{2^b\ge y}
  \frac{1}{(2^b)^{1+\eps}} +  \sum_{3^b\ge y}
  \frac{1}{(3^b)^{1+\eps}} + (2^{-1-\eps}-6^{-1-\eps}) \sum_{s^b\ge y}
    \frac{1}{(s^b)^{1+\eps}} \right].
\ee
The first two sums in \eqref{ceps} total $\le \frac72 y^{-1-\eps}$.
To estimate the third sum, let $S(t)$ denote the number of proper
prime powers $\le t$.  By Theorem 1 and Corollay 1 of \cite{RS}, we
have
$$
\frac{x}{\log x} \le \pi(x) \le \frac{x}{\log x} \( 1 + \frac{3}{2\log
  x} \) \qquad (x\ge 17).
$$
If $t\ge 10^{10}$, then $S(t) > \pi(t^{1/2}) \ge \frac{2 t^{1/2}}{\log
  t}$ and
\begin{align*}
S(t) &= \sum_{k\ge 2} \pi(t^{1/k}) \le \sum_{k=2}^7 \pi(t^{1/k}) +
  \( \frac{\log t}{\log 2} - 7 \) \pi(t^{1/8}) \\
&\le \sum_{k=2}^7 \frac{k t^{1/k}}{\log t} \( 1 + \frac{3k}{2\log t}\)
  +\( \frac{\log t}{\log 2} - 7 \) \frac{8 t^{1/8}}{\log t} \(1 +
  \frac{12}{\log t}\) \\
&\le 2.43 \frac{t^{1/2}}{\log t}.
\end{align*}
By partial summation,
\be\label{sumsb}
\begin{split}
\sum_{s^b\ge y} \frac{1}{(s^b)^{1+\eps}} &= -
\frac{S(y^{-})}{y^{1+\eps}} + (1+\eps)\int_y^\infty
\frac{S(t)}{t^{2+\eps}}\, dt \\
&\le - \frac{2}{y^{1/2+\eps}\log y} + \frac{2.43(1+\eps)}{\log y}
\int_y^\infty \frac{dt}{t^{3/2+\eps}} \\
&= \frac{0.43 + \frac{2.43}{1+2\eps}}{y^{1/2+\eps}\log y}.
\end{split}
\ee
Combined with \eqref{ceps}, this completes the proof.
\end{proof}

\begin{lem}\label{fws}
Let $p$ be a prime and $p^{a+1} \ge 10^{10}$.  Then 
$$
\# \{ q\le x : p^a \| (q-1), f(q) \ge p^{a+1}\} \le
\frac{x}{p^{\frac{3a+1}{2}} \log(p^{a+1})} \left[ 2.86 +
  c(\eps)(1+1/\eps) \frac{x^\eps}{p^{(2a+1)\eps}} \right].
$$
\end{lem}

\begin{proof}  
If $p^a \| (q-1)$ and $f(q) \ge p^{a+1}$, then either
$p^a s^b | (q-1)$ for some proper prime power $s^b$ with $s\ne p$ and
$s^b\ge p^{a+1}$, or there is a prime $r|(q-1)$ with $f(r)\ge p^{a+1}$.
The number of such $q\le x$ is, using Lemma \ref{fqlarge},
\eqref{sumsb} and partial summation,
\begin{align*}
&\le \sum_{s^b \ge p^{a+1}} \frac{x}{p^a s^b} +
  \sum_{\substack{r\le x/p^a \\ f(r) \ge p^{a+1}}} \frac{x}{p^a r} \\
&\le \frac{2.86 x}{p^{(3a+1)/2}\log (p^{a+1})} +  c(\eps)
  \frac{x}{p^{a+(1/2+\eps)(a+1)}\log(p^{a+1})}  \left[
  \pfrac{x}{p^a}^{\eps}  + \int_{p^{a+1}}^{x/p^a} u^{-1+\eps}\, du
  \right].
\end{align*}
This completes the proof of the lemma.
\end{proof}

\begin{proof}[Proof of Theorem \ref{thm3}]
Let $p^a \ge \max(10^{10},K_1)$, $x=p^{aK_2}$ and
$\eps=\frac{1}{2K_2}$.  By Lemmas \ref{pix} and \ref{fws},
\begin{align*}
\# \{ q\le x : p^a \| (q-1), f(q) < p^{a+1} \} &= \pi(x;p^a,1) -
\pi(x;p^{a+1},1) \\
&\qquad\qquad  - \# \{ q\le x : p^a\|(q-1), f(q) \ge p^{a+1} \} \\
&\ge  K_3 \frac{x/\log x}{p^{a+1/2}\log p} - c'(\eps)
\frac{x}{p^{\frac{3a+1}{2}} \log(p^{a+1})} p^{(K_2-2)a\eps} \\
&> 0
\end{align*}
if $p^a$ is large enough, where $c'(\eps)$ is a constant 
depending only on $\eps$.
\end{proof}

%%%%%%%%%%%%%%%%%%%%%%%%%%%%%%%%%%%%%%%%%%%%%%%%%%%%%%
%
%
\section{Proof of Theorem \ref{thm:ERH}}\label{sec:ERH}
%
%
%%%%%%%%%%%%%%%%%%%%%%%%%%%%%%%%%%%%%%%%%%%%%%%%%%%%%%

We first take care of small $p^a$.  If $a=1$ and $p\le 18000000$
($1151367$ primes) and when $a\ge2$ and $p^a \le 10^{10}$ ($10084$ prime
powers), we find a prime $q$ with $p^a \| (q-1)$ and
$q<p^{2a+1}$.  By \eqref{fqp}, $f(q)<p^{a+1}$ for such $q$.  The
calculations were performed using PARI/GP.

Next, suppose $a=1$, $p>18000000$ and put $x=p^3$.  By \eqref{ERH},
\begin{align*}
\pi(x;p,1)-\pi(x;p^2,1) &\ge \frac{\text{li}(x)}{p-1} - \sqrt{x} \log (xp^2)
- \frac{x}{p^2} \\
&\ge \frac{p^2}{\log p} \left[ \frac13 - 5 \frac{\log^2 p}{p^{1/2}} -
  \frac{\log p}{p} \right] > 0,
\end{align*}
as desired.

Lastly, suppose $a\ge 2$ and $p^a > 10^{10}$, and put $x=p^{3a}$.  By
\eqref{ERH}, 
\be\label{pixp}
\begin{split}
\pi(x;p^a,1)-\pi(x;p^{a+1},1) &\ge \frac{\text{li}(x)}{p^a} - \sqrt{x}
\log (x^2 p^{4a+2}) \\
&\ge \frac{p^{2a}}{\log(p^a)} \left[ \frac13 - 11
  \frac{\log^2(p^a)}{p^{a/2}} \right] \\
&\ge 0.275 \frac{p^{2a}}{\log(p^a)}.
\end{split}
\ee
Since we may take $C(\frac14)=7.37$ in Lemma
\ref{chains}, we have
$c(\frac14) \le 22$ for Lemma \ref{fqlarge}.  
By Lemma \ref{fws} and \eqref{pixp},
\begin{align*}
\# \{ q\le x : p^a \| (q-1), f(q)<p^{a+1} \} &\ge 0.275
  \frac{p^{2a}}{\log(p^a)} - \frac{p^{\frac{3a-1}{2}}}{\log(p^{a+1})}
  \left[ 2.86 + 110 p^{\frac{a-1}{4}} \right] \\
&\ge\frac{p^{2a}}{\log(p^a)} \left[ 0.275 - \frac{2.03}{p^{a/2}} -
    \frac{66}{p^{a/4}} \right] \\
&>0,
\end{align*}
as desired.

\medskip

{\bf Acknowledgements.}  We thank the anonymous referee for useful suggestions.
 This work started during a visit of the second author at the
 Mathematics Department of the University of Illinois in
 Urbana-Champaign  in January of 2007. He thanks the people of 
that department for their hospitality.

%%%%%%%%%%%%%%%%%%%%%%%%
%
%  References
%
%%%%%%%%%%%%%%%%%%%%%%%%

\end{document}